\documentclass[12pt,a4paper]{amsart}

\usepackage[utf8]{inputenc}
\usepackage[english]{babel}

\usepackage{latexsym}
\usepackage{amsfonts}
\usepackage{amsmath}
\usepackage{amssymb}
\usepackage{graphicx}
\usepackage{color}
\usepackage{enumerate}

\usepackage{amsthm}

\newcommand{\R}{\mathbb{R}}

\newcommand{\C}{\mathbb{C}}

\newcommand{\ov}[1]{\overline{#1}}

\renewcommand\Re{\operatorname{Re}}
\renewcommand\Im{\operatorname{Im}}
\newcommand\supp{\operatorname{supp}}

\newcommand\p{\operatorname{\partial}}

\newtheorem{thm}{Theorem}[section]

\newtheorem{lem}[thm]{Lemma}
\newtheorem{prop}[thm]{Proposition}

\theoremstyle{definition}

\numberwithin{equation}{section}

\begin{document}

\title[]{A uniqueness result for an inverse problem of the steady state convection-diffusion equation}
\author[]{Valter Pohjola}
\date{}
\address{Department of Mathematics and Statistics, Helsingin yliopisto / Helsingfors universitet / University of Helsinki, Finland}
\email{valter.pohjola@helsinki.fi}

\begin{abstract} 
We consider the inverse boundary value problem for the steady state 
convection diffusion equation. We prove that a velocity field $V$, is uniquely determined by the 
Dirichlet-to-Neumann map, when $V \in C^{0,\gamma} (\Omega)$, $2/3< \gamma \leq 1$, i.e.
when $V$ is a H\"older continuous vector field with $2/3< \gamma \leq 1$.
\end{abstract}

\maketitle

\section{Introduction}

\noindent The steady state convection-diffusion equation 
\begin{align}\label{eq:conv_prob}
             (-\Delta + V \cdot  \nabla) u &= 0, \quad \text{in}\quad \Omega, \\
                          u|_{\p \Omega}  &= f, \nonumber
\end{align}
can be seen as a time independent model for 
transport phenomena in a fluid due to a diffusion 
process and convection caused by the fluid velocity $V$. One specific model
is heat transfer in a fluid, in which case $u$ is taken as the temperature.
In the following we will consider this problem  assuming 
that\footnote{Here $H^s(\Omega)$ refers to the $L^2$ based Sobolev space with smoothness index $s$.
}
$f \in H^{1/2}(\p \Omega)$ and $V \in C^{0,\gamma}(\Omega,\R^n)$, with $2/3 < \gamma \leq 1$ and
where the set $\Omega \subset \R^n$, $n \geq 3$ will be a bounded open set with Lipschitz boundary.
Recall that the space of H\"older continuous functions, $C^{0,\gamma}(\Omega)$, 
$0 < \gamma \leq 1$ is defined as 
\begin{align*}
    C^{0,\gamma}(\Omega) =  \Big\{ g \in C(\ov{\Omega}) \;:\;
    |g|_{C^{0,\gamma}(\Omega)} :=  \sup_{x,y \in \Omega, x\neq y}\frac{|g(x)-g(y)|}{|x-y|^\gamma} < \infty \Big\},
\end{align*}
equipped  with the norm
\begin{align*}
    \|g\|_{C^{0,\gamma}(\Omega)} :=  \|g\|_{L^\infty(\Omega)} + |g|_{C^{0,\gamma}(\Omega)}.
\end{align*}

A physical formulation of the inverse problem we are about to consider, is to think of 
$u$ as the temperature in the region $\Omega$, we then ask if it is possible to determine the velocity field $V$ 
in the region $\Omega$ by controlling the temperature on the boundary and by measuring the
heat flux on the boundary. 

The boundary measurements are mathematically modeled by the 
so called Dirichlet to Neumann map (DN-map for short). 
This is the map $\Lambda_V$ taking $f$ to $\p_n u := (n \cdot \nabla u)|_{\p\Omega}$,
where $n$ is the outward pointing unit normal to $\p \Omega$.
The unique solvability of the Dirichlet problem  \eqref{eq:conv_prob} in $H^1(\Omega)$ 
(see  Theorems 8.1 and 8.3 in \cite{GT}) shows that the DN-map well defined. 
The normal derivative  $\p_n u$ needs, however in this case to be understood in a distributional sense,
because of the non-smooth solutions we consider.
The DN-map can then be defined in a weak sense, as the operator
$\Lambda_V : H^{1/2}(\p\Omega) \to H^{-1/2}(\p\Omega)$ given by
\begin{align*}
\langle \Lambda_V f,\varphi \rangle := \int_{\Omega} (\nabla u \cdot \nabla \phi + V\cdot \nabla u\, \phi)dx,
\end{align*}
where $L_V u:= (-\Delta + V \cdot  \nabla) u =0$, in $\Omega$, $u|_{\Omega} = f$ 
and $\varphi \in H^{1/2}(\p \Omega)$, 
$\phi \in H^1(\Omega)$, with $\phi |_{\p\Omega} = \varphi$. Here $\langle \cdot,\cdot \rangle$ 
denotes the distribution
duality on $\p\Omega$. 
Notice also that the definition is independent of the choice of an extension
$\phi$ of $\varphi$.

The mathematical form of the inverse problem is then the question, if the DN-map of
the Dirichlet problem \eqref{eq:conv_prob} determines the velocity field $V$.
The main result of this paper is the following theorem.

\begin{thm} \label{thm1} 
    Let $V_j \in C^{0,\gamma}(\Omega,\R^n)$, $j=1,2$ with $2/3 < \gamma \leq 1$.
    Assume that $\Lambda_{V_1}=\Lambda_{V_2}$, then $V_1 = V_2$ in $\Omega$.
\end{thm}

\bigskip
The first uniqueness result for the above inverse problem was given 
by Cheng, Nakamura and Sommersalo in \cite{CNS}, where they prove the 
unique determination of the velocity field $V$, for $V \in C^{\infty}(\ov{\Omega})$, and
$\p \Omega \in C^\infty$.
Salo improved this in \cite{Sa}, where it is shown that
the result also holds when $V$ is Lipschitz continuous, i.e. $V \in C^{0,1}(\Omega)$.
This was in turn improved by Knudsen and Salo
in \cite{KS} where they prove that $V$ can be any H\"older continuous function 
provided that $\nabla \cdot V \in L^\infty$. Theorem \ref{thm1} improves on this by showing that 
the restriction $\nabla \cdot V \in L^\infty$, is unnecessary for 
H\"older continuous vector fields $V \in C^{0,\gamma}(\Omega)$, when $2/3 < \gamma \leq 1$.

The inverse problem of the closely related magnetic Schr\"odinger equation, was first studied by
Sun in \cite{S}. There have been several improvements of this result by various authors. The 
sharpest and most recent result is given by Krupchyk and Uhlmann in \cite{KU} where they prove that
the inverse problem is solvable for an electric potential $q\in L^\infty$ and  a magnetic 
potential $A \in L^\infty$.

A first remark on Theorem \ref{thm1} concerns its relations to the 
celebrated Calderon problem (see e.g. \cite{SyU}). The Calderon problem asks if one can determine the 
conductivity in the interior of an object by measuring the current on the boundary, when one
controls the voltage on the boundary (or vice versa), or in more mathematical terms if the 
DN-map corresponding to a Dirichlet problem of the conductivity equation
$\nabla \cdot (\sigma \nabla u) = 0,$
where $\sigma$ is the conductivity, determines the conductivity.
Writing the conductivity equation in non-divergence form we get that
\[ 
\Delta u + \nabla \log (\sigma )\cdot \nabla u  = 0 . \]
This shows that the  \eqref{eq:conv_prob} is a more general and therefore a more 
difficult problem then the Calderon problem.

As a second remark on Theorem \ref{thm1} we point out that 
the over all method of proving Theorem \ref{thm1} is to reduce it to an inverse problem
for the magnetic Schrödinger equation, which is a self-adjoint first order perturbation of the Laplacian.
We will more specifically be utilizing the method of proving uniqueness for the inverse 
problem of the magnetic Schr\"odinger equation given in  \cite{KU}.
One of the main ideas is that one can still use the methods of
\cite{KU} for electric potentials with  worse regularity of a specific distributional form,
provided one assumes that the magnetic potentials are more regular.

The paper is organized as follows. In section 2 we reduce Theorem \ref{thm1} to a claim about the magnetic
Schr\"odinger operator. Section 3 is devoted to constructing complex geometric optics solutions. 
In section 4 we prove the unique determination of the magnetic field and in section 5 we 
prove the unique determination of the electric potential.

\section{Reduction to the Magnetic Schr\"odinger case}
The purpose of this section is to reduce Theorem \ref{thm1} to a similar statement 
concerning the magnetic Schr\"odinger operator.
The argument is formulated by  Cheng, Nakamura and Sommersalo in \cite{CNS} 
and by Salo in \cite{Sa}.
The magnetic Schr\"odinger operator is formally given by
\begin{align*}
L_{A,q}u  = -\Delta u -iA\cdot \nabla u -i \nabla \cdot (Au)+ (A^2+q)u.
\end{align*}
We are going to consider the case where $A\in C^{0,\gamma}(\Omega,\R^n)$ 
and $q=\nabla \cdot F + p$, with $F\in C^{0,\gamma}(\Omega,\R^n)$  and $p \in L^\infty(\Omega, \C)$. 
Hence we need to understand $L_{A,q}$ in a distributional sense, as an operator
$L_{A,q}:H^1(\Omega)\to H^{-1}(\Omega)$, given by
\[
    \langle L_{A, q} \phi,\psi \rangle := \int_{\Omega} \nabla \phi \cdot \nabla \psi + i
    A \cdot (\phi \nabla \psi - \psi \nabla \phi) + (A^2 + p)\phi \psi - F  \cdot \nabla(\phi \psi ) \, dx,
\]
where $ \phi \in H^1(\Omega) $ and $ \psi \in H^1_0(\Omega) $.

The inverse problem for the magnetic Schr\"odinger operator
we are about to consider comes from the Dirichlet Problem
\begin{align*}
    L_{A,q} u &= 0, \quad \text{in}\quad \Omega, \\
    u|_{\p\Omega} &= f,
\end{align*} 
where $f$ is in the Sobolev space $H^{1/2}(\p \Omega)$.  
The normal component of the magnetic gradient on the boundary, $(\p_n+i n\cdot A)u|_{\p \Omega}$,
here $n$ denotes the outward pointing unit normal vector on $\partial \Omega$, is
in our case defined, following \cite{KU}, as the bounded linear map 
$N_{A, q} : H^1(\Omega) \to H^{-1/2}(\p\Omega)$ given by
\[ 
    \langle N_{A,q} u, \varphi \rangle = \int_{\Omega} \nabla u \cdot \nabla \phi + i A
\cdot (u \nabla \phi - \phi \nabla u) + (A^2+p) u \phi  - F \cdot \nabla (u \phi) \, dx 
\]
for any $ u \in H^1(\Omega) $ such that $ L_{A, q}u = 0 $ and any $ \varphi \in
H^{1/2}(\p\Omega) $, such that $ \phi|_{\p \Omega} = \varphi $. 
The definition is independent of the choice of an extension
$\phi$ of $\varphi$.

We shall consider the more general notion of a Cauchy data set, instead of the DN-map 
when dealing with the magnetic Schr\"odinger equation.
The Cauchy data sets are the sets of boundary data of solutions, i.e.
\[
    C_{A,q}:=\{(u|_{\p \Omega},N_{A,q}u): u\in H^1(\Omega)\textrm{ and }  L_{A,q}u=0 \textrm{ in } \Omega\}.
\]

The magnetic field corresponding to a potential $A$ is given by the 2-form $dA$, which is 
defined as
\begin{align} \label{eq:magDef}
    dA=\sum_{1\le j<k\le n}(\p_{j}A_k-\p_{k}A_j)dx_j\wedge dx_k,
\end{align}
this definition should be understood in the sense of non-smooth differential forms 
(a.k.a. currents).

Our aim is now to reduce Theorem \ref{thm1} to the following Proposition, after which the
rest of the paper is devoted to proving this Proposition.

\begin{prop} \label{thm2} Let $\Omega \subset \R^n$ be a bounded domain with Lipschitz 
    boundary.
    Assume that $A_1,A_2,F_1,F_2 \in C^{0,\gamma}(\Omega,\R^n)$, $2/3 < \gamma \leq 1$, with
    $A_1=A_2$ and $F_1=F_2$ on $\p\Omega$, and let $p_1,p_2 \in L^\infty(\Omega,\C)$.
    Assume that $C_{A_1,q_1}=C_{A_2,q_2}$, then $dA_1 = dA_2$ and
    $\nabla \cdot F_1 + p_1 = \nabla \cdot F_2 + p_2$ in $\Omega$.
\end{prop}

The above result is a variation of the main result in \cite{KU}. It differs from this
by being applicable to lower regularity electric potentials (i.e. of the special distributional form),
but it also by requires more regularity on the magnetic potentials.

Another more general  point concerning the above result is that, we cannot in general
hope to recover the magnetic potential $A$. This is because of the gauge invariance 
of the Cauchy data sets. If $\psi\in C^{1,\gamma}(\Omega)$  and $\psi|_{\p \Omega}=0$, then  
$C_{A,q}=C_{A+\nabla\psi,q}$, i.e. it 
is possible to change the magnetic potentials without disturbing the boundary data 
(see Proposition \ref{prop_gauge_1} in the appendix).

At several points we will need extensions of H\"older continuous functions to a larger set 
containing $\Omega$. The following basic extension result on H\"older  continuous functions 
will be used for this (see Theorem 3 on page 174 in \cite{St} and Theorem 16.11 
on page 342 in \cite{CDK}).

\begin{lem}\label{HoldExt} 
Let $\Omega \subset \R^n$ be open set with Lipschitz boundary.
Then there exists a continuous linear extension operator $E$,
\[ 
    E: C^{0,\gamma}(\Omega) \to C^{0,\gamma}_0(\R^n), 
\]
for $0\leq \gamma\leq1$. More precisley there exists a constant $C=C(\Omega)>0$,
such that for every $f \in C^{0,\gamma}(\Omega)$, $\supp(E(f))$ is compact,
\[ 
    E(f)|_\Omega = f 
\]
and one has the norm estimate
\[ 
    \| E(f) \|_{C^{0,\gamma}(\R^n)}  \leq C \| f \|_{C^{0,\gamma}(\Omega) }.
\]
\end{lem}

We will also need the following boundary reconstruction result from \cite{Sa}
(see Theorem 1.9 in \cite{Sa}). 
\begin{thm} \label{bndry_rec}
Let $\Omega \subset \R^n$ be open set with Lipschitz boundary and $n\geq 3$. Assume 
$V_1,V_2\in C^{0,\gamma}(\Omega,\R^n)$, $0<\gamma \leq 1$. If $\Lambda_{V_1} = \Lambda_{V_2}$,
then $V_1|_{\p\Omega}=V_2|_{\p\Omega}$.
\end{thm}

Next we show how Theorem \ref{thm1} follows from Proposition \ref{thm2}.
We follow the argument given in \cite{Sa}. The rest of the paper will focus on proving Proposition
\ref{thm2}.

\bigskip
\textit{Proof of Theorem \ref{thm1}.}
By Theorem \ref{bndry_rec} we know that $V_1=V_2$ on $\p \Omega$.
Lemma \ref{HoldExt}  allows us then to extend $V_j$ to a 
ball $B$, $\Omega \subset\subset B$ so that $V_j \in C^{0,\gamma}(B,\R^n)$, $V_j|_{\p B} = 0$
and $V_1=V_2$ on $B \setminus \Omega$.
Lemma \ref{lem_Cauchy_data_conv} below
shows that the  above extension does not 
alter the DN-maps, i.e. $\Lambda^B_{V_1}=\Lambda^B_{V_2}$.
We may thus assume that that $\Omega = B$ and that $V_1=V_2=0$ on $\p\Omega=\p B$.

We now consider the magnetic Schr\"odinger operators $L_{A_j,q_j}$, $j=1,2$ that 
coincide with $L_{V_j}$.  That is we choose
\[
    A_j := iV_j/2 \quad\text{and}\quad q_j := V_j^2/4 - \nabla \cdot V_j/2, 
\]
which gives that $L_{A_j,q_j} = L_{V_j}$.

Next we want to show that 
$C_{A_j,q_j} = \{ (f,\Lambda_{V_j}f) \,|\, f \in  H^{1/2}(\p B) \}$. We need only to 
show that $N_{A_j,q_j} u_j = \Lambda_{V_j} u_j$, $j=1,2$. 
Let $ u_j \in H^1(B) $ be such that $ L_{A_j, q_j}u_j = 0 $ and assume that $ \varphi \in
H^{1/2}(\p B) $ and that $\phi \in H^1(B)$ is an extension of $\varphi$, 
i.e. $ \phi|_{\p B} = \varphi$. Then  
by definition and because $V_j=0$ on $\p B$
\begin{align*}
    \langle N_{A_j,q_j} u_j, \varphi \rangle 
    &= \int_{B} (\nabla u_j \cdot \nabla \phi - \frac{1}{2}V_j
    \cdot (u_j \nabla \phi - \phi \nabla u_j) 
    + \frac{1}{2} V_j \cdot  \nabla (u_j \phi) )\, dx \\
    &= \int_{B} (\nabla u_j \cdot \nabla \phi  + V_j \cdot \nabla u_j \phi )\, dx \\
    &= \langle \Lambda_{V_j} u_j, \varphi \rangle.
\end{align*}
The assumption that $\Lambda_{V_1}=\Lambda_{V_2}$, implies therefore that
$C_{A_1,q_1} = C_{A_2,q_2}$.

We can now apply Proposition \ref{thm2},
which gives that $dV_1 = dV_2$. By the Poincar\'{e} Lemma (see Theorem 8.3 in \cite{CDK}), 
there exists an $\psi \in C^{1,\gamma}(B)$, s.t.  $V_1-V_2 = \nabla \psi$, 
since $\nabla \psi = 0$ outside $\supp(V_1) \cup \supp(V_2)$, we have that $\psi$ is constant
near $\p B$. We may hence add a constant to $\psi$, so that $\psi = 0$ near $\p B$.

The second consequence of Proposition \ref{thm2} is  that $q_1 = q_2$, so that
$V_1^2/2 - \nabla \cdot V_1 = V_2^2/2 - \nabla \cdot V_2$. This together with
the fact that $V_2 = \nabla \psi- V_1$, gives the equation
\begin{align} \label{eq:quasilin}
    \Delta \psi - V_1 \cdot \nabla \psi + \frac{1}{2} (\nabla \psi)^2 = 0 \text{ in } B,
\end{align}
Next we prove that $\psi \in C^2(B)$. Because of \eqref{eq:quasilin} we have that
$\psi \in C^0(\ov{B})$ satisfies 
\[
    \Delta \psi  = f \text{ in }  B,
\]
with $f = V_1 \cdot \nabla \psi - \frac{1}{2} (\nabla \psi)^2 \in C^{0,\gamma}(B)$.
By interior Schauder estimates (see Theorem 7.18 in \cite{Zw}) we know that 
$\psi \in C^{2,\gamma}(\ov{V})$, for every open $V \subset \subset B$.
It follows that $\psi \in C^2(B)$.

We may now apply the maximum principle to $\psi$ (see Theorem 10.1 in \cite{GT}).
From this it follows that $\psi = 0$ in $B$, since $\psi|_{\p B}=0$. 
We may thus conclude that $V_1=V_2$.
\begin{flushright} 
    $\Box$
\end{flushright}

\section{Complex geometric optics solutions and remainder estimates} \label{sec:cgo}

\noindent In this  section we shortly review the construction of complex geometric optics (CGO for 
short) solutions and then derive some remainder estimates related to these.
We follow by large the construction given in \cite{KU}. 
We are however dealing with more regular magnetic potentials, which allows us to get the better 
remainder estimates that are needed. This and the more irregular electric potentials require
us to make some modifications to the argument in \cite{KU}.

Smooth approximations of the potentials will be an important tool in the following.
Our smoothing procedure will consist of an extension followed by a convolution with a mollifier.
More specifically, given an $A \in C^{0,\gamma}(\Omega,\C^n)$, we consider 
an open bounded set $\Omega'$, s.t.  $\Omega \subset\subset \Omega'$. 
By Lemma \ref{HoldExt} there is an extension of $A$ to $\R^n$, 
$A'\in C^{0,\gamma}(\R^n,\C^n)$, s.t. $A=A'$ in $\Omega$, $A'|_{\R^n \setminus \Omega'} = 0$ and
\begin{align} \label{es:HoldContExt}
    \|A'\|_{C^{0,\gamma}(\R^n,\C^n)}  \leq C \| A \|_{C^{0,\gamma}(\Omega,\C^n)}.
\end{align}
Moreover let $ \Psi$ belong to $ C^\infty_0(\R^n) $ with $ 0 \leq \Psi(x) \leq 1 $ for
all $ x \in \R^n $, $ \supp \Psi \subset \{ x \in \R^n : |x| \leq 1 \} $ and $
\int_{\R^n} \Psi \, dx = 1 $. Define $\Psi_\theta (x) = \theta^n \Psi(\theta x)
$ for $\theta \in (0, \infty)$ and $ x \in \R^n $. We define  $A^\sharp$ for any 
$A' \in C_0^{0,\gamma}(\R^n,\C^n)$, as
\[
     A^\sharp := \Psi_\theta \ast A'.
\]
Notice also that \eqref{es:HoldContExt}  implies that
$\|A^\sharp\|_{C^{0,\gamma}(\R^n,\C^n)} \leq C \| A \|_{C^{0,\gamma}(\Omega,\C^n)}$,
where $C$ is independent of $\theta$.

The following Lemma gives some basic and well known estimates for the above approximation scheme 
(see \cite{H}).
\begin{lem}\label{LemApprox} 
Assume that $A \in C^{0,\gamma}(\Omega,\C^n)$, with $0<\gamma \leq 1$ 
and let $A'$ be the above extension of $A$  to $\R^n$.
Then
\begin{align}
    \| A'-A^\sharp \|_{L^\infty(\R^n,\C^n)}  &\leq C \theta^{-\gamma}, \label{es:Asharp1} \\
   \| \p^{\alpha} A^\sharp \|_{L^\infty(\R^n,\C^n)}  &\leq C \theta^{|\alpha|-\gamma}, \label{es:Asharp2} 
\end{align}
as $\theta \to \infty$, for any multi-index $\alpha$, with $|\alpha| \geq 1$.
\end{lem}    
\begin{proof} Let $\Psi$ be as above. Assume that $x \in \R^n$.
    For the first estimate we use \eqref{es:HoldContExt} and have that 
\begin{align*}
    |A'(x) - A^\sharp(x) | &= 
   \big | \int_{\R^n} A'(x) \Psi(y)\,dy -  \int_{\R^n} A'(x-y) \theta^n \Psi(\theta y) \,dy\big| \\
   &\leq 
   \int_{\R^n} | A'(x) \Psi(y) - A'(x-y/\theta) \Psi(y)| \,dy \\
   &\leq C \| A \|_{C^{0,\gamma}(\Omega,\C^n)}
   \theta^{-\gamma} \int_{\R^n} |y|^{\gamma}  |\Psi(y)| \,dy \\
   &\leq 
   C \theta^{-\gamma}.
\end{align*}
To derive the second estimate \eqref{es:Asharp2} notice firstly that
\begin{align*}
    \int_{\R^n} \p^\alpha \Psi(y) dy = 0,
\end{align*}
for all multi indexes $\alpha$, with $|\alpha| \geq 1$. Let $x \in \R^n$, then
using the above observation, we have that
\begin{align*}
    |\p^\alpha A^\sharp(x)| 
    &= 
    \big| \int_{\R^n} A'(y)  \theta^{n+|\alpha|} (\p^\alpha \Psi) \big(\theta(x-y)\big) \,dy \big| \\
    &= 
    \big| \int_{\R^n} A'(x-y/\theta)  \theta^{|\alpha|} (\p^\alpha \Psi) (y) \,dy \big| \\
    &= 
    \big| \int_{\R^n} \big( A'(x-y/\theta) - A'(x) \big) \theta^{|\alpha|} (\p^\alpha \Psi) (y) \,dy \big| \\
    & \leq \| A \|_{C^{0,\gamma}(\Omega,\C^n)} \theta^{|\alpha|}
    \int_{\R^n} |y/\theta|^{\gamma} \big| (\p^\alpha \Psi) (y) \big| \,dy  \\
    & \leq C \theta^{|\alpha|-\gamma}.
\end{align*}
\end{proof}

\textbf{Remark.} In the rest of this section we will consider $A$ to be extended as $A'$ outside $\Omega$,
i.e. we use $A$ to denote the extension $A'$.

\bigskip
We will now show how to construct so called complex geometric optics solutions
following the argument in \cite{KU}. It is natural to formulate this 
in terms of certain  semiclassical norms that are defined as follows
\begin{align*}
&\|u\|_{H^1_{\textrm{scl}}(\Omega)}^2  := \|u\|_{L^2(\Omega)}^2+\|h\nabla u\|_{L^2(\Omega)}^2,\\
&\|v\|_{H^{-1}_{\textrm{scl}}(\Omega)} := \sup_{0\ne \psi\in
C_0^\infty(\Omega)}\frac{|\langle
v,\psi\rangle_{\Omega}|}{\|\psi\|_{H^1_{\textrm{scl}}(\Omega)}}.
\end{align*} 

The construction of CGO solutions is based
on the solvability result below. The solvability  result is in turn a consequence of 
a perturbed Carleman estimate, Proposition \ref{PCE} in the appendix.
The argument that shows how to obtain the
solvability result from the Carleman estimate is standard and we refer 
to the proof of Proposition 2.3 in \cite{KU}.

\begin{prop} \label{solvability} 
    Let $A,F \in L^\infty(\Omega, \C^n)$, $p \in L^\infty(\Omega,\C)$ and
    $q= \nabla \cdot F + p$. Furthermore let $\varphi(x) = \alpha \cdot x$, $\alpha \in \R^n$ with $|\alpha|=1$.
    If $h>0$ is small enough, then  for any $v \in H^{-1}(\Omega)$, there is a solution of the 
    equation
    \[
        e^{\varphi/h}h^2L_{A,q}(e^{-\varphi/h}u ) = v , \text{ in } \Omega,
    \]
    which satisfies 
    \begin{align} \label{eq:solv_est}
        \|u\|_{H^1_{\emph{scl}}(\Omega)}\le \frac{C}{h}\|v\|_{H^{-1}_{\emph{scl}}(\Omega)}.
    \end{align}
\end{prop}
\bigskip

The CGO solutions $u\in H^1(\Omega)$ considered here solve
\[ L_{A, q} u = 0, \]
with $ A ,F \in C^{0,\gamma}(\Omega, \C^n) $, $0<\gamma \leq1$, $p \in L^\infty(\Omega,\C)$ 
and have the form
\begin{equation} \label{eq:CGOform}
u(x;\zeta,h) = e^{x\cdot\zeta/h} (a(x;\zeta,h) + r(x;\zeta,h)),
\end{equation}
where $\zeta\in\C^n$ with $\zeta\cdot\zeta=0$ and $|\zeta|\sim 1$; $ h $ is a
small semiclassical parameter; $a$ is a smooth amplitude and $r$ is a reminder
term. 
%

We begin by assuming that $\zeta \in \C^n$, $\zeta = \zeta_0 + \zeta_1 $ is such that
\begin{align} \label{eq:zetaAssum}
 &\zeta\cdot\zeta=0,\;
    \zeta_0 \text{ is constant with respect to $h$, } \zeta_1=\mathcal{O}(h),\\  
    &\text{as } h\to 0 \text{ and }|\Re\zeta_0| =|\Im \zeta_0|=1. \nonumber
\end{align}
Abbreviate the conjugated operator multiplied by $h^2$, with 
\[
L_\zeta:= e^{- \zeta \cdot x / h} h^2 L_{A,q} ( e^{\zeta \cdot x / h}).
\]
Then in order to construct $ u(\cdot; \zeta, h) $ of the
form \eqref{eq:CGOform}, it is enough to prove the existence of a
$ r(\cdot; \zeta, h) \in H^1(\Omega) $ solving
\begin{equation}\label{eq:remainder}
L_\zeta r = -L_\zeta a,
\end{equation}
in $\Omega$ for a suitable $a$.
The $ a \in C^\infty(\R^n) $ is picked as the solution to 
\begin{equation}\label{eq:transport}
    \zeta_0 \cdot \nabla a + i \zeta_0 \cdot A^\sharp a = 0, \quad \text{in} \quad \R^n,
\end{equation}
so that left hand side of \eqref{eq:remainder} becomes, using \eqref{eq:zetaAssum}, \eqref{eq:transport}
and \eqref{eq:conjMaMf} given below,
\begin{align} \label{eq:aparts}
-L_\zeta a= & h^2\Delta a + i h^2 A \cdot \nabla a - h^2 m_A (a) - h^2 (A^2+p) a + 2h \zeta_1 \cdot \nabla a \\
    & + 2hi\zeta_0 \cdot(A-A^\sharp)a
    + 2 hi \zeta_1 \cdot A a - h^2 m_{\nabla \cdot F} (a). \nonumber
\end{align}
Here $m_A $ and $m_{\nabla \cdot F}$ are the bounded linear operators from $ H^1(\Omega) $ to $
H^{-1}(\Omega) $ defined by
\begin{align*}
    \langle m_A (\phi), \psi \rangle &:= \int_{\Omega} i \phi A \cdot \nabla \psi \, dx, \\
    \langle m_{\nabla \cdot F} (\phi), \psi \rangle &:=  -\int_{\Omega} F \cdot \nabla (\phi\psi) \, dx,
\end{align*}
for all $ \phi \in H^1(\Omega) $ and all $ \psi \in H^1_0(\Omega)$.
It easy to see that
\begin{align} \label{eq:conjMaMf}
   e^{-\zeta\cdot x/h}\circ h^2m_A \circ e^{\zeta\cdot x/h} &= -hi\zeta\cdot A+h^2 m_A, \\
   e^{-\zeta\cdot x/h}\circ h^2m_{\nabla \cdot F} \circ e^{\zeta\cdot x/h} &= h^2 m_{\nabla \cdot F}.
   \nonumber 
\end{align}

If we look for solutions to \eqref{eq:transport} in the form
$a = e^{\Phi^\sharp}$,
it will be enough that $ \Phi^\sharp(\cdot; \zeta_0, \theta) $ satisfies
\begin{equation}
\zeta_0 \cdot \nabla \Phi^\sharp + i \zeta_0 \cdot A^\sharp = 0 \label{eq:deltabarSHARP}
\end{equation}
in $ \R^n $. The fact that $ \mathrm{Re}\, \zeta_0 \cdot \mathrm{Im}\, \zeta_0 = 0 $
and $ |\mathrm{Re}\, \zeta_0| = |\mathrm{Im}\, \zeta_0| = 1 $, implies that
$N_{\zeta_0}:= \zeta_0 \cdot \nabla $ 
is a $\overline{\partial}-$operator in suitable coordinates.
The Cauchy operator $N_{\zeta_0}^{-1}$, defined by 
\[
    (N_{\zeta_0}^{-1}f)(x) := \frac{1}{2\pi} \int_{\R^2} 
    \frac{f(x-y_1 \Re \zeta_0 - y_2 \Im \zeta_0)}{y_1+iy_2}\, dy_1 dy_2, 
\]
for $f \in C_0(\R^n),$ is the inverse of the $\overline{\partial}-$operator
 and gives thus that
\[
\Phi^\sharp = N_{\zeta_0}^{-1} (- i \zeta_0 \cdot A^\sharp) \in C^\infty (\R^n).
\]
We will also use the following basic continuity result for the
Cauchy operator (see \cite{Sa}, Lemma 7.4).

\begin{lem} \label{Cauchy_op}
Let $f\in W^{k,\infty}(\R^n)$, $k \geq 0$, with $\supp(f) \subset B(0,R)$. 
Then  we have that
\begin{align}
\label{eq:Cauchy_op_cont}
\|N_{\zeta_0}^{-1} f \|_{W^{k,\infty}(\R^n)}\le C\|f\|_{W^{k,\infty}(\R^n)}, 
\end{align}
where $C=C(R)$.  
\end{lem}
\bigskip

Using now Lemma \ref{LemApprox} and Lemma \ref{Cauchy_op}, we have that 
\begin{align} \label{es:phi_sharp}
\| \p^\alpha \Phi^\sharp\|_{L^\infty(\R^n)} \leq C \theta^{|\alpha|-\gamma}
\end{align}
for $ \theta \in (1, \infty) $ and a multi-indexes $\alpha$, $|\alpha| \geq 1$.
Moreover, defining $ \Phi (\cdot; \zeta_0) := (\zeta_0 \cdot
\nabla)^{-1} (- i \zeta_0 \cdot A) \in L^\infty (\R^n) $,
solves analogously
\begin{equation}
\zeta_0 \cdot \nabla \Phi + i \zeta_0 \cdot A = 0 \label{eq:deltabar}
\end{equation}
and  satisfies
\begin{align}
    \| \Phi (\cdot; \zeta_0)\|_{L^\infty(\R^n)} \leq C \|A\|_{L^\infty(\R^n)}. \label{es:phi} 
\end{align}
Lemma \ref{Cauchy_op} and estimate \eqref{es:Asharp1}
imply that the functions $\Phi^\sharp$ converge to $\Phi$ in $L^\infty(\Omega)$ or more
explicitly that
\begin{align*}
 \big \| \Phi^\sharp (\cdot, \zeta_0, \theta) - \Phi (\cdot; \zeta_0) \big \|_{L^\infty(\R^n)} 
 \leq C \theta^{-\gamma}. 
\end{align*}

With the $a$ at hand the solvability result, Proposition \ref{solvability} 
guarantees the existence of a solution $r$, to equation \eqref{eq:remainder}, such 
that
\begin{align}\label{es:remNorm}
    \|r\|_{H^1_{\textrm{scl}}(\Omega)}\le \frac{C}{h}\| L_\zeta a \|_{H^{-1}_{\textrm{scl}}(\Omega)}.
\end{align}
Now we determine how the left hand side of the above estimate depends on $h$,
i.e. we estimate the $H^{-1}_{\textrm{scl}}(\Omega)$-norm of the terms in equation \eqref{eq:aparts}.
This gives us the behaviour of the $H^1_{\emph{scl}}(\Omega)$-norm of the remainder term
$r$ in the parameter $h$. 

Let $0\ne \psi\in C_0^\infty(\Omega)$. Then using \eqref{es:phi_sharp}, the fact that
$\zeta_1=\mathcal{O}(h)$ and the Cauchy--Schwarz inequality we get that
\begin{align*}
&|\langle h^2\Delta a, \psi \rangle_\Omega|\le \mathcal{O} (h^2\theta^{2-\gamma})
    \|\psi\|_{L^2(\Omega)}\le \mathcal{O} (h^2\theta^{2-\gamma})
    \|\psi\|_{H^{1}_{\textrm{scl}}(\Omega)}, \\
    &|\langle ih^2 A\cdot \nabla a,\psi\rangle_\Omega|\le \mathcal{O} (h^2\theta^{1-\gamma})
    \|\psi\|_{H^{1}_{\textrm{scl}}(\Omega)},\\
& |\langle 2h\zeta_1\cdot \nabla a,\psi\rangle_\Omega|\le \mathcal{O} (h^2\theta^{1-\gamma})
    \|\psi\|_{H^{1}_{\textrm{scl}}(\Omega)},\\
& |\langle 2hi\zeta_1\cdot Aa,\psi \rangle_\Omega|\le \mathcal{O} (h^2)
    \|\psi\|_{H^{1}_{\textrm{scl}}(\Omega)},\\
& |\langle h^2 (A^2+p) a,\psi\rangle_\Omega|\le \mathcal{O} (h^2)
    \|\psi\|_{H^{1}_{\textrm{scl}}(\Omega)}.
\end{align*}
By Lemma \ref{LemApprox} we have on the other hand that
\begin{align*}
|\langle 2hi\zeta_0 \cdot(A-A^\sharp)a,\psi\rangle_\Omega |&\le
    \mathcal{O}(h)\|a\|_{L^\infty(\Omega)}\|A-A^\sharp\|_{L^2(\Omega)}\|\psi\|_{L^2(\Omega)}\\
    &\le \mathcal{O}(h)\theta^{-\gamma} \|\psi\|_{H^1_{\textrm{scl}}(\Omega)}.
\end{align*}
Again by  Lemma \ref{LemApprox}  and estimate  \eqref{es:phi_sharp} we have that
\begin{align*}
|\langle h^2m_A(a),& \psi \rangle_\Omega|\le \bigg| \int_\Omega ih^2 A^\sharp
    a\cdot \nabla \psi dx\bigg| + \bigg|\int_\Omega i h^2 (A-A^\sharp) a\cdot \nabla \psi
    dx\bigg|\\
&\le \bigg| \int_\Omega ih^2 (\nabla \cdot (A^\sharp a)) \psi dx\bigg|
    +\mathcal{O}(h)\|A-A^\sharp\|_{L^2(\Omega)}\|h\nabla\psi\|_{L^2(\Omega)}\\
    &\le (\mathcal{O}(h^2\theta^{1-\gamma}) +\mathcal{O}(h)\theta^{-\gamma}) 
    \|\psi\|_{H^1_{\textrm{scl}}(\Omega)}. 
\end{align*}
Similarly with the help of Lemma \ref{LemApprox}  and estimate \eqref{es:phi_sharp} we have that
\begin{align*}
    |\langle h^2m_{\nabla\cdot F}(a),& \psi \rangle_\Omega|
    \le \bigg| \int_\Omega ih^2 F^\sharp \cdot \nabla (a\psi) dx\bigg| 
    + \bigg|\int_\Omega i h^2 (F-F^\sharp) \cdot \nabla (a \psi) dx\bigg|\\
    &\le \bigg| \int_\Omega ih^2 \nabla F^\sharp \cdot a\psi dx\bigg| 
    + \bigg|\int_\Omega i h^2 (F-F^\sharp) \cdot \nabla (a \psi) dx\bigg|\\
    &\le  Ch^2 \theta^{1-\gamma}\|\psi\|_{H^{1}_{\textrm{scl}}(\Omega)}
    +C h^2\|F-F^\sharp\|_{L^2(\Omega)} \|\nabla a\|_{L^\infty(\Omega)} \|\psi\|_{L^2(\Omega)}\\
    &\quad\quad\quad\quad\quad\quad\quad\quad\quad\quad
    +C h\|F-F^\sharp\|_{L^2(\Omega)} \|a\|_{L^\infty(\Omega)} \|h\nabla\psi\|_{L^2(\Omega)}\\
    &\le C(h^2\theta^{1-\gamma}+h^2\theta^{1-2\gamma}+h\theta^{-\gamma})
    \|\psi\|_{H^1_{\textrm{scl}}(\Omega)}. 
\end{align*}
Combining the above estimates gives that
\begin{align*}
    \| L_\zeta a \|_{H^{-1}_{\emph{scl}}(\Omega)} \le C (h^2\theta^{2-\gamma} + h \theta^{-\gamma})
\end{align*}
By choosing $\theta = h^{-1/2}$, we get hence by estimate \eqref{es:remNorm} that
\begin{align*} 
    \|r\|_{H^1_{\emph{scl}}(\Omega)}\le  C h^{\gamma/2}
\end{align*}
We have thus derived the following  Proposition.
\begin{prop} \label{CGOest}
Let $\Omega\subset \R^n$, $n\ge 3$, be a bounded open set with Lipschitz boundary.  
Let $A, F \in C^{0,\gamma}(\Omega,\R^n)$, $0 < \gamma \leq 1$,
$p \in L^\infty(\Omega,\C^n)$, with $q:= \nabla \cdot F + p$
and let $\zeta\in \C^n$ satisfy \eqref{eq:zetaAssum}. 
Then for all $h>0$ small enough, there
exists a solution $u(x,\zeta;h)\in H^1(\Omega)$ of
\[L_{A,q}u=0, \text{ in } \Omega \]
of the form
$u(x,\zeta;h)=e^{x\cdot\zeta/h}(e^{\Phi^\sharp(x,\zeta_0;h)}+r(x,\zeta;h))$.
The function  $\Phi^\sharp (\cdot,\zeta_0;h)\in C^\infty(\R^n) \cap L^\infty(\R^n) $ satisfies 
\begin{align} \label{es:phi_h_est}
    \|\p^\alpha \Phi^\sharp\|_{L^\infty(\R^n)}\le C_\alpha h^\frac{\gamma-|\alpha|}{2},
\end{align}
for all $\alpha$, $|\alpha|\ge 1$, and $\Phi^\sharp (\cdot,\zeta_0;h)$
converges in the $L^\infty$-norm  to $\Phi(\cdot,\zeta_0):=N_{\zeta_0}^{-1}(-i\zeta_0\cdot A)\in
L^\infty(\R^n)$. More precisely 
\begin{align} \label{es:PhiPhiSharp}
 \big \| \Phi^\sharp (\cdot, \zeta_0, h) - \Phi (\cdot; \zeta_0) \big \|_{L^\infty(\R^n)} 
 \leq C h^{\gamma/2}. 
\end{align}
The remainder $r$ is such that
\begin{align} \label{es:r_h_est}
    \|r\|_{H^1_{\emph{scl}}(\Omega)}\le  C h^{\gamma/2},
\end{align}
as $h\to 0$. 
\end{prop}

\section{Uniqueness of the magnetic field} \label{sec:magUniq}

\noindent
This section contains a proof of the first part of Proposition \ref{thm2}, i.e.
we show that $dA_1 = dA_2$. 
We begin by stating an integral identity, which readily follows from the assumption
that $C_{A_1,q_1}=C_{A_2,q_2}$. The proof
can be found in \cite{KU} and only minor modifications are needed to make it work
with electric potentials used here.
\begin{prop}
\label{int_identity}
Let $\Omega\subset \R^n$, $n\ge 3$,  be a bounded open set with Lipschitz boundary.
Assume that $p_1,p_2\in L^\infty(\Omega,\C)$ and 
$A_1,A_2,F_1,F_2\in C^{0,\gamma}(\Omega,\C^n)$, with $0 < \gamma \leq 1$.
If $C_{A_1,q_1}=C_{A_2,q_2}$,
then the following integral identity 
\begin{align}
\label{eq:intId}
\int_\Omega i&(A_1-A_2)\cdot (u_1\nabla \overline{u_2}-\overline{u_2}\nabla u_1)
            + (A_1^2-A_2^2+p_1-p_2)u_1\overline{u_2}  \nonumber \\
            -&(F_1-F_2) \cdot (u_1\nabla \overline{u_2}+\overline{u_2}\nabla u_1)\,dx=0 
\end{align}
holds for any $u_1,u_2\in H^1(\Omega)$ satisfying $L_{A_1,q_1}u_1=0$ in
$\Omega$ and $L_{\ov{A_2},\ov{q_2}}u_2=0$ in $\Omega$,
respectively.  
\end{prop}
\bigskip

The idea is then to choose specific CGO solutions and insert them into the integral identity 
and then show that this reduces, in the limit $h \to 0$ to a specific Fourier transform.
The CGO will be chosen as follows.   
Let $\xi,\mu_1,\mu_2\in\R^n$ be such that $|\mu_1|=|\mu_2|=1$ and
$\mu_1\cdot\mu_2=\mu_1\cdot\xi=\mu_2\cdot\xi=0$. Define
\begin{align} \label{eq_zeta_1_2}
\zeta_1 &=\frac{ih\xi}{2}+\mu_1 + i\sqrt{1-h^2\frac{|\xi|^2}{4}}\mu_2 , \nonumber \\
\zeta_2 &=-\frac{ih\xi}{2}-\mu_1+i\sqrt{1-h^2\frac{|\xi|^2}{4}}\mu_2,
\end{align}
so that $\zeta_j\cdot\zeta_j=0$, $j=1,2$, and
\begin{align} \label{eq:z1_plus_z2}
    (\zeta_1+\ov{\zeta_2})/h=i\xi. 
\end{align}
Here $h>0$ is a small enough.
Moreover, $\zeta_1= \mu_1+ i\mu_2+\mathcal{O}(h)$ and $\zeta_2= -\mu_1+ i\mu_2+\mathcal{O}(h)$ as $h\to 0$. 

For all $h>0$, that are small enough there
exists, by Proposition \ref{CGOest}  a solution $u_1(x,\zeta_1;h)\in H^1(\Omega)$ to the 
equation $L_{A_1,q_1}u_1=0$ in $\Omega$, of the form
\begin{equation}
\label{eq_u_1}
u_1(x,\zeta_1;h)=e^{x\cdot\zeta_1/h}(e^{\Phi_1^\sharp(x,\mu_1+i\mu_2;h)}+r_1(x,\zeta_1;h)),
\end{equation}
where $\Phi_1^\sharp(\cdot,\mu_1+i\mu_2;h) \in C^\infty(\R^n) \cap L^\infty(\R^n)$ is given by 
\begin{equation}
\label{eq_phi_1_sharp_def}
\Phi_{1}^\sharp(\cdot,\mu_1+i\mu_2;h):=N_{\mu_1+i\mu_2}^{-1}
\big(-i(\mu_1+i\mu_2)\cdot A_1^\sharp\big)
\end{equation}
and $\Phi_1^\sharp(\cdot,\mu_1+i\mu_2;h) \to \Phi_1(\cdot,\mu_1+i\mu_2)$
in $L^\infty(\R^n)$ as $h\to 0$, where $\Phi_1$ is given by 
Proposition \ref{CGOest}. 

Similarly, for all $h>0$ small enough, there exists a solution
$u_2(x,\zeta_2;h)\in H^1(\Omega)$ to the equation
$L_{\overline{A_2},\overline{q_2}}u_2=0$ in $\Omega$, of the form
\begin{equation}
\label{eq_u_2}
u_2(x,\zeta_2;h)=e^{x\cdot\zeta_2/h}(e^{\Phi_2^\sharp(x,-\mu_1+i\mu_2;h)}+r_2(x,\zeta_2;h)),
\end{equation}
where $\Phi_2^\sharp(\cdot,-\mu_1+i\mu_2;h) \in C^\infty(\R^n) \cap L^\infty(\R^n)$ is given by 
\begin{equation}
\label{eq_phi_2_sharp_def}
\Phi_{2}^\sharp(\cdot,-\mu_1+i\mu_2;h):=N_{-\mu_1+i\mu_2}^{-1}
\big(-i(-\mu_1+i\mu_2)\cdot \ov{A_2^\sharp}\big)
\end{equation}
and $\Phi_2^\sharp(\cdot,-\mu_1+i\mu_2;h) \to \Phi_2(\cdot,-\mu_1+i\mu_2)$
in $L^\infty(\R^n)$ as $h\to 0$, where $\Phi_2$ is given by 
Proposition \ref{CGOest}. 

Notice also that we have by estimates \eqref{es:phi_h_est} and \eqref{es:r_h_est},
of Proposition \ref{CGOest}, that 
\begin{align}
    \|\nabla \Phi^\sharp_j\|_{L^\infty(\R^n)} &\leq C h^\frac{\gamma-1}{2}, \label{es:phi_h}\\ 
           \|r_j\|_{H^1_{\emph{scl}}(\Omega)} &\leq  C h^{\gamma/2}, \label{es:r_h}
\end{align}
for $j=1,2$.

The next step is to insert the $u_1$ and $u_2$ specified above into \eqref{eq:intId}, multiply 
by $h$ and let $h \to 0$, in an attempt to obtain a Fourier transform of the magnetic field.
This is done in the next Lemma. The proof is based on the argument found in \cite{KU}. The difference is 
however in how the electric potential is estimated. The crucial observation is that 
the last term in \eqref{eq:intId} containing the electric potentials,
goes to zero, in $h$ when multiplied with an extra factor of $h$,
even though it closely resembles the first term with the magnetic potentials, for which this 
does not happen.

\begin{lem}\label{LemTempFourier}  For $A_1,A_2,\mu_1,\mu_2$ and $\xi$ as above we have that
\begin{equation}
\label{eq:with_phases_R_n}
(\mu_1+i\mu_2)\cdot\int_{\R^n} (A_1-A_2) e^{ix\cdot\xi} e^{\Phi_{1}+\ov{\Phi_{2}}}dx=0.
\end{equation}
\end{lem}
\begin{proof} We use the abbreviations $A:=A_1-A_2$, $F:=F_1-F_2$ and $p:=p_1-p_2$. 
First we multiply \eqref{eq:intId} by $h$. For the non-gradient terms in \eqref{eq:intId}
we have by \eqref{es:r_h} that
\begin{align*}
\Big|
h \int_\Omega &(A_1^2 -A_2^2+p)u_1\overline{u_2}\,dx \Big|  
\\ &= \Big| h \int_\Omega (A_1^2-A_2^2+p) e^{ix\cdot\xi}(e^{\Phi_1^\sharp+\ov{\Phi_2^\sharp}}
+e^{\Phi_1^\sharp}\ov{r_2}+r_1e^{\ov{\Phi_2^\sharp}}+r_1\ov{r_2})
\,dx \Big|  \\
    &\leq Ch \| A_1^2-A_2^2+p\|_{L^\infty}
    \Big( 
    \|e^{\Phi_1^\sharp+\ov{\Phi_2^\sharp}}\|_{L^\infty}
    + \|e^{\Phi_1^\sharp}\|_{L^\infty} \|\ov{r_2}\|_{L^2}  \\
    &\quad\quad\quad\quad\quad\quad\quad\quad\quad\quad
    + \|r_1\|_{L^2} \|e^{\ov{\Phi_2^\sharp}}\|_{L^\infty} 
    + \|r_1\|_{L^2}\| \ov{r_2}\|_{L^2} \Big) \\
    &\leq C h \to 0,
\end{align*}
as $h \to 0$. For our specific CGO solutions, $u_1$ and $u_2$, we hence have that 
\begin{align} \label{eq:modIntId}
h \Big| \int_\Omega 
 iA \cdot (u_1\nabla \overline{u_2}-\overline{u_2}\nabla u_1)\,dx
- h \int_\Omega 
F \cdot (u_1\nabla \overline{u_2}+\overline{u_2}\nabla u_1)
\,dx   \Big| =  \mathcal{O}(h),
\end{align}
as  $h \to 0$.

We continue by estimating the first integral in \eqref{eq:modIntId}. 
Since the solutions $u_1$  and $u_2$ are of the CGO form one gets the following by expanding 
\begin{align} \label{eq:u1du2}
hu_1\nabla\ov{u_2}=&\ov{\zeta_2}e^{ix\cdot\xi}(e^{\Phi_1^\sharp+\ov{\Phi_2^\sharp}}
+e^{\Phi_1^\sharp}\ov{r_2}+r_1e^{\ov{\Phi_2^\sharp}}+r_1\ov{r_2})\\
&+he^{ix\cdot\xi}(e^{\Phi_1^\sharp}\nabla e^{\ov{\Phi_2^\sharp}} +
e^{\Phi_1^\sharp}\nabla \ov{r_2} + r_1\nabla e^{\ov{\Phi_2^\sharp}}
+r_1  \nabla \ov{r_2}) \nonumber.
\end{align}
The first term in the first parantheses in \eqref{eq:u1du2} gives 
\begin{align}\label{lim:first}
\ov{\zeta_2} \cdot\int_\Omega iA
e^{ix\cdot\xi}e^{\Phi_1^\sharp+\ov{\Phi_2^\sharp}}dx\to
-(\mu_1+i\mu_2)\cdot\int_\Omega iA
e^{ix\cdot\xi}e^{\Phi_1+\ov{\Phi_2}}dx.
\end{align} 
as $h\to 0$.  This is because $\ov{\zeta_2}=-\mu_1-i\mu_2+\mathcal{O}(h)$  and
by \eqref{es:PhiPhiSharp} we have that
\begin{align*}
\bigg| (\mu_1+i\mu_2)\cdot\int_\Omega A
e^{ix\cdot\xi}\big(e^{\Phi_1^\sharp+\ov{\Phi_2^\sharp}}-e^{\Phi_1+\ov{\Phi_2}}\big)dx \bigg|
&\leq C\big\|e^{\Phi_1^\sharp+\ov{\Phi_2^\sharp}}- e^{\Phi_1+\ov{\Phi_2}}\big\|_{L^\infty(\Omega)}\\
&\leq C h^{\gamma/2} \to 0
\end{align*}
as $h\to 0$. 
For the next three terms in \eqref{eq:u1du2}, we can use estimate \eqref{es:r_h} and Cauchy--Schwarz
to conclude that
\begin{align} \label{lim:middle}
\bigg|\int_\Omega & iA \cdot \overline{\zeta_2}
    e^{ix\cdot\xi}(e^{\Phi_1^\sharp}\overline{r_2}+r_1e^{\overline{\Phi_2^\sharp}}+r_1\overline{r_2})dx\bigg| 
    \nonumber\\
    &\le C\| A \|_{L^\infty}
    (\big\|e^{\Phi_1^\sharp}\big\|_{L^2}\|\overline{r_2}\|_{L^2}+\|r_1\|_{L^2}\big\|
    e^{\overline{\Phi_2^\sharp}}\big\|_{L^2}+\|r_1\|_{L^2}\|\overline{r_2}\|_{L^2}) \\
    & \leq C h^{\gamma/2} \to 0, \nonumber
\end{align}
as $h\to 0$. For the last part of \eqref{eq:u1du2} containing the factor $h$, we have 
using estimates \eqref{es:r_h} and \eqref{es:phi_h}   that
\begin{align} \label{lim:last}
\bigg|\int_\Omega h iA \cdot e^{ix\cdot\xi}(e^{\Phi_1^\sharp}\nabla
e^{\overline{\Phi_2^\sharp}} + e^{\Phi_1^\sharp}\nabla \overline{r_2} +
r_1\nabla e^{\overline{\Phi_2^\sharp}} +r_1  \nabla \overline{r_2})dx\bigg|\\
\le C h \big( h^{(\gamma-1)/2}+ h^{-1}h^{\gamma/2}+ h^{\gamma/2}h^{(\gamma-1)/2}+h^{\gamma}h^{-1} \big)
 \to 0, \nonumber
\end{align}
as $h\to 0$. 
Expanding the $\ov{u_2}\nabla u_1$ term in \eqref{eq:modIntId} gives
\begin{align} \label{eq:u2du1}
h\ov{u_2}\nabla u_1 = &\zeta_1 e^{ix\cdot\xi}(e^{\Phi_1^\sharp+\ov{\Phi_2^\sharp}}
+e^{\Phi_1^\sharp}\ov{r_2}+r_1e^{\ov{\Phi_2^\sharp}}+r_1\ov{r_2})\\
&+he^{ix\cdot\xi}(\nabla e^{\Phi_1^\sharp}e^{\ov{\Phi_2^\sharp}} +
 \nabla e^{\Phi_1^\sharp} \ov{r_2} + \nabla r_1 e^{\ov{\Phi_2^\sharp}}
+\nabla r_1 \ov{r_2}) \nonumber.
\end{align}
Again $-\zeta_1=-\mu_1-i\mu_2+\mathcal{O}(h)$. The terms in \eqref{eq:u1du2} and \eqref{eq:u2du1}
are of the same form. Doing the analogous estimates for  \eqref{eq:u2du1} gives then that 
\begin{align*}
h \int_\Omega 
 iA \cdot (u_1\nabla \overline{u_2}-\overline{u_2}\nabla u_1) \,dx
 \to
-2i(\mu_1+i\mu_2)\cdot\int_{\R^n} A e^{ix\cdot\xi} e^{\Phi_{1}+\ov{\Phi_{2}}}dx,
\end{align*}
as  $h \to 0$. 

We end the proof by showing that
\begin{align} \label{lim:Fpart}
h \int_\Omega F \cdot (u_1\nabla \overline{u_2}+\overline{u_2}\nabla u_1)
d x  \to 0,
\end{align}
as  $h \to 0$. Using \eqref{eq:u1du2} and \eqref{eq:u2du1} gives that
\begin{align} \label{eq:ududuu}
h(u_1\nabla \overline{u_2}+\overline{u_2}\nabla u_1) \;=\; 
&(\ov{\zeta_2}+\zeta_1) e^{ix\cdot\xi} 
\big(e^{\Phi_1^\sharp+\ov{\Phi_2^\sharp}}
+e^{\Phi_1^\sharp}\ov{r_2}+r_1e^{\ov{\Phi_2^\sharp}}+r_1\ov{r_2} \big)\nonumber \\
&+he^{ix\cdot\xi} \big(
e^{\Phi_1^\sharp}\nabla e^{\ov{\Phi_2^\sharp}} 
+ e^{\Phi_1^\sharp}\nabla \ov{r_2} + r_1\nabla e^{\ov{\Phi_2^\sharp}}
+ r_1  \nabla \ov{r_2}\\ 
&\quad\quad\quad\quad
+\nabla e^{\Phi_1^\sharp}e^{\ov{\Phi_2^\sharp}} 
+ \nabla e^{\Phi_1^\sharp} \ov{r_2} + \nabla r_1 e^{\ov{\Phi_2^\sharp}}
+ \nabla r_1 \ov{r_2} 
\big). \nonumber
\end{align}
The second term on the right hand side is of the same form as the second term
on the right hand side of \eqref{eq:u1du2} and \eqref{eq:u2du1}. The contribution of these
terms are therefore zero in the limit $h \to 0$. 

For the first term on the right hand side of  \eqref{eq:ududuu} we get, 
using \eqref{eq:z1_plus_z2} and \eqref{es:r_h}, the estimate
\begin{align*}
\bigg|  h \int_{\Omega} &i\xi \cdot F ( e^{ix\cdot\xi} 
(e^{\Phi_1^\sharp+\ov{\Phi_2^\sharp}}
+e^{\Phi_1^\sharp}\ov{r_2}+r_1e^{\ov{\Phi_2^\sharp}}+r_1\ov{r_2})
    ) \,dx\bigg| \\
    &\leq \mathcal{O}(h)(1+h^{\gamma/2} + h^{\gamma/2}  + h^{\gamma} ) \to 0,
\end{align*}
as $h\to0$.
This shows that \eqref{lim:Fpart}
holds.
\end{proof}

It turns out that the $e^{\Phi_{1}+\ov{\Phi_{2}}}$ term can be dropped from \eqref{eq:with_phases_R_n}. 
This is guaranteed by Proposition 3.3 in \cite{KU} (see also \cite{ER} and \cite{S}). 
Using the abbreviation $A:=A_1-A_2$ we thus obtain
\begin{align*}
    (\mu_1+i\mu_2) \cdot \int_{\R^n} A e^{ix\cdot\xi} dx = (\mu_1+i\mu_2)\cdot\widehat{A}(-\xi) = 0,
\end{align*}
where $\widehat{A}$ stands for the Fourier transform of $A$.
Moreover for any $\mu \in \R^n$, with $\mu \cdot \xi = 0$, we have therefore that 
$\mu \cdot \widehat{A} =0$. 
It follows that  the Fourier transform of the component functions of
\eqref{eq:magDef} are zero. To see this notice that the above implies that 
\begin{align*}
    \xi_j \widehat{A_k} - \xi_k \widehat{A_j} =
    (\xi_j e_k - \xi_k e_j) \cdot \widehat{A} = 0,
\end{align*}
since $\xi \cdot (\xi_j e_k - \xi_k e_j) = 0$, where $e_k$ denote the standard basis
vectors of $\R^n$. We have thus proved that $dA_1 = dA_2$.
\\

\textbf{Remark.} Notice that, we only need the condition $0<\gamma \leq 1$ 
in recovering the magnetic potentials,  instead of $2/3<\gamma \leq 1$.

\section{Uniqueness of the electric potential}
To finish the proof of Proposition \ref{thm2}, 
we need to show that $q_1 = \nabla \cdot F_1 +p_1 = \nabla \cdot F_2+p_2=q_2$.
Lemma \ref{HoldExt} and the assumption that $A_1=A_2$, $F_1=F_2$ on $\p \Omega$ and that
$\p \Omega$ is Lipschitz,
allows us to extend $A_j$ and $F_j$, $j=1,2$ to a ball $B$,
with $\ov{\Omega} \subset B$, so  that $A_1 = A_2$ and $F_1 = F_2$ in $B\setminus\Omega$,
$F_j = A_j = 0$ on $\p B$ and $A_j,F_j \in C^{0,\gamma}(B)$, for $j=1,2$.

In the previous section we proved that  $d(A_1 -A_2) = 0$. The Poincar\'{e} Lemma
implies now that there is a $\psi \in C^{1,\gamma}(B)$ s.t.
$A_1-A_2 = \nabla \psi$ in $B$ (see \cite{CDK}). We can moreover choose $\psi$ so 
that $\psi|_{\p B} = 0$, since $A_1=A_2=0$ in $B \setminus \Omega$.
By Lemma \ref{lem_Cauchy_data} and Proposition \ref{prop_gauge_1} below, we have that
\[
    C_{A_1,q_1}^{B}=C_{A_2,q_2}^{B}
    =C_{A_2+\nabla\psi,q_2}^{B}=C_{A_1,q_2}^{B}.
\]
Proposition \ref{int_identity} gives then that
\begin{align} \label{eq:redIntId}
    \int_B (-F \cdot \nabla(u_1 \ov{u_2}) + p u_1\ov{u_2})\,dx = 0,
\end{align}
for any $u_1,u_2 \in H^1(B)$, satisfying $L_{A_1,q_1}u_1=0$, $L_{\ov{A_2},\ov{q_2}}u_2=0$
in $B$ and where $F := F_1-F_2$ and $p:=p_1-p_2$.

We now suppose, as in section \ref{sec:magUniq} that $u_1$ and $u_2$ are
given by \eqref{eq_u_1} and \eqref{eq_u_2} (when $\Omega = B$), with $A_1=A_2$
and consider the limit of \eqref{eq:redIntId} as $h \to 0$.  
Expanding \eqref{eq:redIntId}, using \eqref{eq:z1_plus_z2} gives 
\begin{align} \label{eq:expIdInt}
    &\int_B -F\cdot i\xi e^{ix\cdot\xi}(e^{\Phi_1^\sharp+\ov{\Phi_2^\sharp}}
    +e^{\Phi_1^\sharp}\ov{r_2}+r_1e^{\ov{\Phi_2^\sharp}}+r_1\ov{r_2}) \,dx \nonumber \\
    &+\int_B -F\cdot e^{ix\cdot\xi} \nabla(e^{\Phi_1^\sharp+\ov{\Phi_2^\sharp}}
    +e^{\Phi_1^\sharp}\ov{r_2}+r_1e^{\ov{\Phi_2^\sharp}}+r_1\ov{r_2})\,dx\\ 
    &+\int_B p u_1 \ov{u_2}\,dx = 0. \nonumber
\end{align}
We begin by showing that the second integral in \eqref{eq:expIdInt} tends to zero,
in the limit $h \to 0$. 

We simplify \eqref{eq:expIdInt} firstly by writing $\widetilde{F} := Fe^{ix\cdot\xi}$.
Notice also that $\widetilde{F} \in C^{0,\gamma}(B)$.
The second simplification comes from the fact that $e^{\Phi_1^\sharp+\ov{\Phi_2^\sharp}} = 1$. To show 
this notice first that the Cauchy operator has the  
following properties 
\begin{align*}
\ov{N_{\zeta}^{-1}f}=N_{\ov{\zeta}}^{-1}\ov{f},\quad  N_{-\zeta}^{-1}f=-N_\zeta^{-1}f.
\end{align*}
Applying these to the definitions \eqref{eq_phi_1_sharp_def} and  \eqref{eq_phi_2_sharp_def}
together with the fact that we are now considering the case with $A_1 = A_2$  yields
\begin{align*}
    \Phi_1^\sharp+\ov{\Phi_2^\sharp} = 
    N^{-1}_{\mu_1+i\mu_2}\big (-i(\mu_1+i\mu_2) \cdot (A_1^\sharp - A_2^\sharp) \big) =0, 
\end{align*}
so that 
\begin{align} \label{eq:simp2}
    e^{\Phi_1^\sharp+\ov{\Phi_2^\sharp}} = 1.
\end{align}

Split the second integral in \eqref{eq:expIdInt} into pieces by taking the absolute value and applying 
the triangle inequality. Consider first the first term of the second integral in \eqref{eq:expIdInt}. 
By \eqref{eq:simp2} we have immediately that
\begin{align} \label{es:Fp1}
    \Big|  \int_B \widetilde{F} \cdot \nabla  e^{\Phi_1^\sharp+\ov{\Phi_2^\sharp}}\,dx  \Big|
    = 0.
\end{align} 
Next we consider the terms $\nabla(e^{\Phi_1^\sharp}\ov{r_2})$ and $\nabla(r_1e^{\ov{\Phi_2^\sharp}})$,
coming from  the second integral in \eqref{eq:expIdInt}. 
Notice firstly that $\widetilde{F}|_{\p B} =0$, since $F|_{\p B} =0$.
Letting $\widetilde{F}^\sharp  := \Psi_\theta * \widetilde{F}$, where $\Psi_\theta$ is 
defined as in the beginning of section \ref{sec:cgo} and using the estimates of Proposition \ref{CGOest}
 and Lemma \ref{LemApprox} we get that
\begin{align} \label{es:Fp2}
    \Big|  \int_B \widetilde{F} \cdot \nabla \big( e^{\Phi_1^\sharp}\ov{r_2}\big) \,dx  \Big|
    &= \Big|  \int_B \widetilde{F} \cdot \big( \nabla e^{\Phi_1^\sharp}\ov{r_2} 
    + e^{\Phi_1^\sharp}\nabla \ov{r_2} \big) \,dx  \Big| \nonumber \\
    &\lesssim
    \| \widetilde{F} \cdot \nabla e^{\Phi_1^\sharp}\|_{\infty} \|\ov{r_2} \|_{2}
    + \Big|  \int_B \widetilde{F} \cdot e^{\Phi_1^\sharp}\nabla \ov{r_2}\,dx  \Big| \nonumber \\
    &\lesssim
    h^{(\gamma-1)/2}h^{\gamma/2}
    + \Big|  \int_B \widetilde{F} \cdot e^{\Phi_1^\sharp}\nabla \ov{r_2}\,dx  \Big| \nonumber \\
    &\lesssim
    h^{\gamma-1/2}
    + \Big|  \int_B \nabla \cdot \big(\widetilde{F}^\sharp e^{\Phi_1^\sharp} \big) \ov{r_2}\,dx  \Big| \\
    &\quad\quad\quad\quad
    + \Big|  \int_B \big(\widetilde{F} - \widetilde{F}^\sharp\big) \cdot e^{\Phi_1^\sharp}\nabla \ov{r_2}\,dx  
    \Big|  \nonumber\\
    &\lesssim
    h^{\gamma-1/2}
    + \theta^{1-\gamma} h^{\gamma/2}
    + \|\widetilde{F} - \widetilde{F}^\sharp\|_{\infty} \| \nabla \ov{r_2} \|_{2}  \nonumber\\
    &\lesssim
     h^{\gamma-1/2} 
     + \theta^{1-\gamma} h^{\gamma/2}
     + \theta^{-\gamma}h^{\gamma/2-1}. \nonumber
\end{align}
The last term from the second integral in \eqref{eq:expIdInt} is handled as follows
\begin{align} \label{es:Fp3}
    \Big| \int_B  \widetilde{F} \cdot \nabla (r_1 \ov{r_2}) \,dx\Big| 
    &\lesssim 
    \Big| \int_B  \nabla \cdot \widetilde{F}^\sharp r_1 \ov{r_2}\,dx \Big| 
    + \Big| \int_B  \big(\widetilde{F} -\widetilde{F}^\sharp\big)  
    \cdot \nabla (r_1 \ov{r_2})\,dx \Big| \nonumber \\
    &\lesssim \|\nabla \cdot \widetilde{F}^\sharp \|_{\infty} \|r_1\|_{2}\|\ov{r_2}\|_{2} \nonumber \\
    &\quad
    +\| \widetilde{F} -  \widetilde{F}^\sharp \|_{\infty} 
      \big( \|\nabla r_1\|_{2}\|\ov{r_2}\|_{2} + \|r_1\|_{2}\|\nabla\ov{r_2}\|_{2} \big) \\
    &\lesssim \theta^{1-\gamma} h^{\gamma/2} h^{\gamma/2} 
      + \theta^{-\gamma}h^{-1}h^{\gamma/2}h^{\gamma/2}  \nonumber \\
    &\lesssim \theta^{1-\gamma} h^{\gamma} + \theta^{-\gamma}h^{\gamma-1}. \nonumber
\end{align}
Combining \eqref{es:Fp1}, \eqref{es:Fp2} and \eqref{es:Fp3}
and then choosing $\theta = h^{-1}$, gives for the second integral in \eqref{eq:expIdInt} 
that
\begin{align*}
    \Big|  \int_B & F\cdot e^{ix\cdot\xi} \nabla(e^{\Phi_1^\sharp+\ov{\Phi_2^\sharp}}
    +e^{\Phi_1^\sharp}\ov{r_2}+r_1e^{\ov{\Phi_2^\sharp}}+r_1\ov{r_2})\,dx \Big|  \\
    &\lesssim
    \theta^{1-\gamma}h^{\gamma/2} + \theta^{-\gamma}h^{\gamma/2-1} + h^{\gamma-1/2}\\
    &=
    2 h^{(3\gamma-2)/2 } + h^{\gamma-1/2}  \to  0,
\end{align*}
as $h \to 0$, since we require that $\gamma > 2/3$.

We now return to the first integral in \eqref{eq:expIdInt}. It can be 
estimated using \eqref{es:r_h}
and the  Cauchy--Schwarz inequality 
as follows
\begin{align*}
    \Big| \int_B &-F\cdot i\xi e^{ix\cdot\xi}(
    e^{\Phi_1^\sharp}\ov{r_2}+r_1e^{\ov{\Phi_2^\sharp}}+r_1\ov{r_2}) \,dx \Big|  \\
    &\lesssim 
    \big \|e^{\Phi_1^\sharp}\big \|_{\infty}
    \big \|\ov{r_2}\big \|_{2}
    +\big \|r_1\big\|_{2} \big \|e^{\ov{\Phi_2^\sharp}}\big\|_{\infty}
    +\big \|r_1\big\|_{2} \big \|\ov{r_2}\big \|_{2} \\
    &\lesssim 
    h^{\gamma/2} \to 0,
\end{align*}
as $h \to 0$.
Estimating the third integral in \eqref{eq:expIdInt} in a simliar fashion
and using \eqref{eq:simp2}
we thus conclude that \eqref{eq:expIdInt} reduces to
\begin{align*}
    \int_B (-F\cdot i\xi e^{ix\cdot\xi} + p e^{ix\cdot \xi}) \,dx  = 0,
\end{align*}
in the limit $h \to 0$. This implies that $\mathcal{F} \big( \nabla \cdot F + p \big)(-\xi) = 0$
in the distributional sense, which 
in turn implies that $0= \nabla \cdot F + p = (\nabla \cdot F_1 + p_1) - (\nabla \cdot F_2 + p_2)$, 
finishing the proof of Proposition \ref{thm2}.

\section{Appendix A -- Gauge invariance and Boundary data} 
Gauge invariance plays  an important role when working with
the magnetic Schr\"odinger equation. Here we state the basic result concerning the 
gauge invariance of the Cauchy data sets. This section also includes two results on when 
the equality of the boundary
data on a smaller set implies the equality of the boundary data on a bigger set.

\begin{prop} \label{prop_gauge_1} 
Let $\Omega\subset \R^n$, $n\ge 3$, be a bounded open set with Lipschitz boundary.  
Assume $A,F \in C^{0,\gamma}(\Omega,\C^n)$, $0 < \gamma \leq 1$, $p \in L^\infty(\Omega, \C)$,
$\psi \in C^{1,\gamma}(\Omega,\C)$  and let $q = \nabla \cdot F + p$ . Then  we have 
    \begin{align} \label{eq:conj_lem}
        e^{-i\psi}\circ L_{A,q }\circ e^{i\psi}=L_{A+\nabla \psi, q}.
    \end{align}
    If furthermore, $\psi|_{\p \Omega}=0$ then 
    \begin{align} \label{eq:conj_lem_2}
        C_{A,q}=C_{A+\nabla\psi,q}. 
    \end{align}
\end{prop} 
\begin{proof}
Let $\psi \in C^{1,\gamma}(\Omega)$. By direct computation we know that
for $L_{A,p}$, we have
\[
        e^{-i\psi}\circ L_{A,p }\circ e^{i\psi}=L_{A+\nabla \psi, p}.
\]
Furthermore we have that
\[
        e^{-i\psi}\circ (\nabla \cdot F) \circ e^{i\psi} = \nabla \cdot F,
\]
since for $u,v \in C^\infty_0(\Omega)$, we have that
\begin{align*}
        \big \langle e^{-i\psi}\circ (\nabla \cdot F) \circ e^{i\psi} u, v \big \rangle
        &= -\int_\Omega F \cdot \nabla ( e^{-i\psi} u e^{i\psi} v)\,dx \\
        &=  -\int_\Omega F \cdot \nabla ( uv)\,dx,
\end{align*}
where $\langle \cdot,\cdot \rangle$ stands for the  distributional duality.
Thus recalling that $q = \nabla \cdot F + p$ it follows that 
\begin{align*}
    e^{-i\psi}\circ L_{A,q }\circ e^{i\psi}=L_{A+\nabla \psi, q},
\end{align*}
which proves \eqref{eq:conj_lem}.

In order to prove \eqref{eq:conj_lem_2}, assume that $\psi|_{\p\Omega}=0$. Let 
$u \in H^1(\Omega)$ be a solution to
\[
    L_{A,q} u = 0,  \text{ in }  \Omega.
\]
By  \eqref{eq:conj_lem} we know that $e^{-i\psi}u \in H^1(\Omega)$ satisfies 
\[
    L_{A + \nabla \psi,q} (e^{-i\psi}u) = 0,  \text{ in }  \Omega.
\]
Moreover we have that $e^{-i\psi}u|_{\p\Omega} = u|_{\p \Omega}$. 
It remains hence to show that 
\[
    N_{A + \nabla \psi,q} (e^{-i\psi}u) =  N_{A,q} u,  \text{ on }  \p \Omega.
\]
To that end let $\varphi \in H^{1/2}(\p\Omega)$ and let $\phi\in H^1(\Omega)$ be such
that $\phi |_{\p\Omega} = \varphi$. Then
\begin{align*}
    \big \langle  &N_{A + \nabla \psi,q} (e^{-i\psi}u) , \varphi \big \rangle 
    = \big \langle  N_{A + \nabla \psi,q} (e^{-i\psi}u) , e^{i\psi}\varphi \big \rangle \\ 
    &\quad=  \int_\Omega \nabla(e^{-i\psi}u) \cdot \nabla(e^{i\psi} \phi) 
    +i(A +\nabla \psi) \cdot (e^{-i\psi}u \nabla(e^{i\psi} \phi) \\
    &\quad\quad-\nabla (e^{-i\psi}u) e^{i\psi} \phi) + ((A + \nabla \psi)^2 +p)u\phi- F\cdot\nabla(u\phi)\,dx \\
    &\quad=  \int_\Omega \nabla u \cdot \nabla \phi
    + iA \cdot (u \nabla \phi -\nabla u  \phi) + (A^2 +p)u\phi\\
    &\quad \quad- F\cdot\nabla(u\phi)\,dx \\
    &\quad= \big \langle  N_{A,q} u , \varphi \big \rangle.
\end{align*}

\end{proof}

The next Lemma is a slight modification of Lemma 4.2 in \cite{Sa}, we include the proof for the convenience 
of the reader. The Lemma shows that two DN-maps that coincide on small set, give two DN-maps 
that coincide on a bigger if we extend the potentials so that they are identical outside the smaller set.

\begin{lem} \label{lem_Cauchy_data_conv}
    Assume that $\Omega, \Omega'\subset \R^n$ be  bounded open sets with Lipschitz
    boundaries, such that $\ov{\Omega}\subset \Omega'$
    and let $V_1,V_2\in L^\infty(\Omega',\C^n)$.
    Denote by $\Lambda_{V_j}^\Omega$ the DN-map corresponding to the Dirichlet problem on the set 
    $\Omega$.
    Assume that 
    $V_1=V_2$ in $\Omega'\setminus\Omega$.
    If  $\Lambda_{V_1}^\Omega=\Lambda_{V_2}^\Omega$ then
    $\Lambda_{V_1}^{\Omega'}=\Lambda_{V_2}^{\Omega'}$.
\end{lem}
\begin{proof} 
Given $u_1' \in H^1(\Omega')$, solving $L_{V_1} u_1' = 0$, in $\Omega'$ we
need to find an $u_2'  \in H^1(\Omega')$ solving $L_{V_2} u_2' = 0$, in $\Omega'$ with 
$u_2'|_{\p \Omega'} = u_1'|_{\p \Omega'} $ and $\p_n u_2'|_{\p \Omega'} =\p_n u_1'|_{\p \Omega'}$.

The function $u_1 := u_1'|_{\Omega}$ solves $L_{V_1} u_1 = 0$, in $\Omega$. Let $u_2\in H^1(\Omega)$ be 
such that
$L_{V_2} u_2 = 0$ in $\Omega$ and $u_2|_{\p \Omega} = u_1|_{\p \Omega}$. We know that 
$\p_n u_2|_{\p \Omega} = \p_n  u_1|_{\p \Omega}$,
since $\Lambda_{V_1}^\Omega=\Lambda_{V_2}^\Omega$. Thus $u_1 - u_2 \in H^1_0(\Omega)$. Define
\begin{align*}
u_2' := u_1' - (u_1 - u_2), \textrm{ in }  \Omega',
\end{align*}
where we extended by $u_1 - u_2$ by zero into $\Omega'$.
Clearly $u_2' \in H^1(\Omega')$,
$u_2'|_{\p \Omega'} = u_1'|_{\p \Omega'} $ and $\p_n u_2'|_{\p \Omega'} =\p_n u_1'|_{\p \Omega'}$.

It remains to check that $L_{V_2} u_2' = 0$, in $\Omega'$ in a weak sense. 
Let $\varphi \in C^\infty_0(\Omega')$, then  
\begin{align*}
    \langle L_{V_2} u_2', \varphi \rangle_{\Omega'} 
    &= \int_{\Omega'} \nabla u_2' \cdot \nabla \varphi + V_2 \cdot \nabla u_2'\varphi \\
    &= \int_{\Omega} \nabla u_2' \cdot \nabla \varphi + V_2 \cdot \nabla u_2'\varphi 
     + \int_{\Omega'\setminus \Omega}  \nabla u_2' \cdot \nabla \varphi + V_2 \cdot \nabla u_2'\varphi \\
    &= \int_{\Omega} \nabla u_2 \cdot \nabla \varphi + V_2 \cdot \nabla u_2\varphi 
     + \int_{\Omega'\setminus \Omega}  \nabla u_1' \cdot \nabla \varphi + V_1 \cdot \nabla u_1'\varphi \\
    &= 
    \int_{\Omega'}  \nabla u_1' \cdot \nabla \varphi + V_1 \cdot \nabla u_1'\varphi \\
    &=
    \langle L_{V_1} u_1', \varphi \rangle_{\Omega'}  \\
    &= 0,
\end{align*}
where we use the fact that $u_2|_{\p \Omega} = u_1|_{\p \Omega}$, 
$u_1 = u_1'|_\Omega$  and $\Lambda^\Omega_{V_1} = \Lambda_{V_2}^\Omega$ 
to get the fourth equality.
\end{proof}

We need a similar result concerning the magnetic Schr\"odinger operator. 

\begin{lem}
\label{lem_Cauchy_data}
Let $\Omega, \Omega'\subset \R^n$ be  bounded open sets with Lipschitz boundaries, such that
$\ov{\Omega} \subset \Omega'$.  Let $A_1,A_2,F_1,F_2\in C^{0,\gamma}(\Omega',\C^n)$,
$0<\gamma \leq 1$, $p_1,p_2 \in L^\infty(\Omega',\C^n)$ and let $q_j:= \nabla \cdot F_j + p_j$.
Denote by
$C_{A_j,q_j}^{\Omega}$ the Cauchy data for $L_{A_j,q_j}$ in the set $\Omega$,
$j=1,2$. 
Assume that 
\begin{equation}
\label{eq_equality_A_q}
A_1=A_2,\;F_1=F_2\quad\textrm{and}\quad p_1=p_2, \quad \textrm{in}\quad \Omega'\setminus\Omega.
\end{equation}
If  $C_{A_1,q_1}^{\Omega}=C_{A_2, q_2}^\Omega$ then 
$C_{A_1,q_1}^{\Omega'}=C_{A_2, q_2}^{\Omega'}$.
\end{lem}
\begin{proof}
Given $u_1' \in H^1(\Omega')$, solving $L_{A_1,q_1} u_1' = 0$, in $\Omega'$ we
need to find an $u_2'  \in H^1(\Omega')$ solving $L_{A_2,q_2} u_2' = 0$, in $\Omega'$ with 
$u_2'|_{\p \Omega'} = u_1'|_{\p \Omega'} $ and 
$N_{A_2,q_2} u_2' = N_{A_1,q_1} u_1'$. This implies that 
$C_{A_1,q_1}^{\Omega'}\subset C_{A_2,q_2}^{\Omega'}$, from which the claim follows.

Let $u_1 := u_1'|_{\Omega}$. Then $L_{A_1,q_1} u_1 = 0$, in $\Omega$. Let $u_2\in H^1(\Omega)$ be 
such that
$L_{A_2,q_2} u_2 = 0$, in $\Omega$ and $u_2|_{\p \Omega} = u_1|_{\p \Omega}$.
Because $C_{A_1,q_1}^{\Omega}=C_{A_2, q_2}^\Omega$, we 
know that $N_{A_2,q_2} u_2 = N_{A_1,q_1} u_1$, on $\p\Omega$.

In particular we have that $\varphi := u_2 - u_1 \in H^1_0(\Omega) \subset H^1_0(\Omega')$. Define
\begin{align*}
u_2' := u_1' + \varphi, \textrm{ in }  \Omega',
\end{align*}
where we extended by $u_1 - u_2$ by zero into $\Omega'$.
Clearly $u_2' \in H^1(\Omega')$,
$u_2'|_{\p \Omega'} = u_1'|_{\p \Omega'} $. We need  thus to check that
$L_{A_2,q_2} u_2' = 0$, in $\Omega'$ and that
$N_{A_2,q_2} u_2' = N_{A_1,q_1} u_1'$.

Let $\psi \in C_0^\infty(\Omega')$, then
\begin{align*}
    \langle L_{A_2,q_2} u_2', \psi \rangle_{\Omega'} = 
    \int_{\Omega'} &\nabla (u_1'+\varphi) \cdot \nabla \psi + i A_2
    \cdot ((u_1'+\varphi) \nabla \psi - \psi \nabla(u_1'+\varphi)) \\ 
    & + (A_2^2+p_2)(u_1'+\varphi)  \psi  - F_2 \cdot \nabla ( (u_1'+\varphi) \psi) \, dx.
\end{align*}
Since $u_1'+\varphi = u_2$ on $\Omega$, we have that
\begin{align*}
    \langle L_{A_2,q_2} u_2', \psi \rangle_{\Omega'} &= 
    \int_{\Omega} \nabla u_2 \cdot \nabla \psi + i A_2
    \cdot (u_2 \nabla \psi - \psi \nabla u_2) \\ 
    &\quad+ (A_2^2+p_2)u_2\psi  - F_2 \cdot \nabla ( u_2 \psi) \, dx  \\
    &+\int_{\Omega'\setminus \Omega} \nabla u_1'   \cdot \nabla \psi + i A_1
    \cdot (u_1' \nabla \psi - \psi \nabla u_1') \\ 
    &\quad+ (A_1^2+p_1)u_1'  \psi  - F_1 \cdot \nabla (u_1' \psi) \, dx  \\
    &+\int_{\Omega'\setminus \Omega} \nabla \varphi \cdot \nabla \psi + i A_1
    \cdot (\varphi\nabla \psi - \psi \nabla\varphi) \\ 
    &\quad+ (A_1^2+p_1)\varphi  \psi  - F_1 \cdot \nabla ( \varphi \psi) \, dx 
\end{align*}
The last integral is zero, since $\supp(\varphi) \subset \Omega$. Hence 
using the assumption that $N_{A_2,q_2} u_2 = N_{A_1,q_1} u_1$, on $\p\Omega$
gives
\begin{align*}
    \langle L_{A_2,q_2} u_2', \psi \rangle_{\Omega'} &= 
    \langle N_{A_2,q_2} u_2, \psi|_{\Omega} \rangle_{\p\Omega} \\
    &\quad+\int_{\Omega'\setminus \Omega} \nabla u_1'   \cdot \nabla \psi + i A_1
    \cdot (u_1' \nabla \psi - \psi \nabla u_1') \\ 
    &\quad\quad+ (A_1^2+p_1)u_1'  \psi  - F_1 \cdot \nabla (u_1' \psi) \, dx  \\
    &=\langle L_{A_1,q_1} u_1', \psi \rangle_{\Omega'} 
    =0.
\end{align*}
Thus we see that $L_{A_2,q_2} u_2' = 0$, in $\Omega'$. 

A similar deduction shows that $N_{A_2,q_2} u_2' = N_{A_1,q_1} u_1'$.
Hence we have that $C_{A_1,q_1}^{\Omega'}\subset C_{A_2,q_2}^{\Omega'}$.
\end{proof}

\section{ Appendix B -- A Carleman estimate}

In this section we prove a Carleman estimate that implies the solvability result Proposition 
\ref{solvability}, in section \ref{sec:cgo}. The proof is a straight forward extension of the one in
\cite{KU}, and we give it here for the convenience of the reader. The main concern is how to incorporate the 
$\nabla \cdot F$ term into the result in \cite{KU}. 

The estimate we are about to prove is a perturbation of  
the Carleman estimate for the Laplacian, given in \cite{STz} (see also \cite{KU}). 
We state this result as follows.

\begin{prop}
Let $\varphi(x) = \alpha \cdot x$, $\alpha \in \R^n$, $|\alpha| = 1$ and let
$\varphi_\varepsilon=\varphi+\frac{h}{2\varepsilon}\varphi^2$.  Then for
$0<h\ll \varepsilon\ll 1$ and $s\in\R$, we have
\begin{align} \label{eq:CE_lap}
\frac{h}{\sqrt{\varepsilon}}\|u\|_{H^{s+2}_{\textrm{scl}}(\R^n)}\le
C\|e^{\varphi_\varepsilon/h}h^2\Delta(e^{-\varphi_\varepsilon/h}u)\|_{H^s_{\textrm{scl}}(\R^n)},
\quad C>0,
\end{align}
for all $u\in C^\infty_0(\Omega)$.  
\end{prop}

We now apply this result in the case $s=-1$ and a fixed $\varepsilon>0$ that is sufficiently small.

\begin{prop} \label{PCE}
Let $\varphi(x) = \alpha \cdot x$, $\alpha \in \R^n$ with $|\alpha| = 1$. Assume
$A,F \in L^\infty(\Omega, \C^n)$, $p \in L^\infty(\Omega,\C)$ and
$q= \nabla \cdot F + p$. Then for  $0<h\ll 1$, we have 
\begin{align} \label{eq:CE_schr}
h\|u\|_{H^{1}_{\textrm{scl}}(\R^n)}\le
C\|e^{\varphi/h}h^2L_{A,q}(e^{-\varphi/h}u) \|_{H^{-1}_{\textrm{scl}}(\R^n)},
\end{align}
for all $u\in C^\infty_0(\Omega)$.  
\end{prop}

\begin{proof}
Let $\varphi_\varepsilon=\varphi+\frac{h}{2\varepsilon}\varphi^2$ be the convexified weight,
with $\varepsilon >0$ and $0<h\ll \varepsilon\ll 1$. Then in the proof Proposition 2.2 in 
\cite{KU}, it is shown that
\begin{align}
\label{eq:AA_est}
\|e^{\varphi_\varepsilon/h} h^2 A\cdot D(e^{-\varphi_\varepsilon/h}u) +
e^{\varphi_\varepsilon/h} h^2 D\cdot( A e^{-\varphi_\varepsilon/h}u)
\|_{H_{\textrm{scl}}^{-1}(\R^n)}\le
\mathcal{O}(h)\|u\|_{H^1_{\textrm{scl}}(\R^n)},
\end{align}
where $D:=i^{-1}\nabla$.
Here the implicit constant depends on $\|A\|_{L^\infty(\Omega)}$, $\| \varphi \|_{L^\infty(\Omega)}$
and $\| D \varphi \|_{L^\infty(\Omega)}$ (see (2.4) in \cite{KU}). 

Furthermore, we have for all $0 \neq \psi \in C^\infty_0(\Omega)$ that
\begin{align*}
    \big| \langle e^{\varphi_\varepsilon/h} h^2 \nabla \cdot F (e^{-\varphi_\varepsilon/h} u) , \psi \rangle \big| 
    &\leq  h^2 \int_{\R^n} | F \nabla \cdot ( e^{\varphi_\varepsilon/h}u e^{-\varphi_\varepsilon/h}\psi )| \\
    &\leq  h \|F\|_{L^\infty(\R^n)} \big( \|h\nabla u\|_{L^2(\R^n)} \|\psi\|_{L^2(\R^n)} \\
    &\quad + \|u\|_{L^2(\R^n)} \|h \nabla \psi \|_{L^2(\R^n)} \big)\\
    &\leq \mathcal{O}(h)\|u\|_{H^1_{\textrm{scl}}(\R^n)}\|\psi\|_{H^1_{\textrm{scl}}(\R^n)}.
\end{align*}
It follows from the definition of the $H^{-1}_{\textrm{scl}}$-norm that 
\begin{align} \label{eq:Fterm_est}
    \|  e^{\varphi_\varepsilon/h} h^2 \nabla \cdot F
    (e^{-\varphi_\varepsilon/h} u)\|_{H^{-1}_{\textrm{scl}}(\R^n)} 
    &\leq \mathcal{O}(h)\|u\|_{H^1_{\textrm{scl}}(\R^n)}.
\end{align}
By choosing a small fixed $\varepsilon>0$ that is independent of $h$, we conclude
from estimates \eqref{eq:CE_lap},\eqref{eq:AA_est} and \eqref{eq:Fterm_est} that
\begin{align*}
\big \|&e^{\varphi_\varepsilon/h}(-h^2\Delta)(e^{-\varphi_\varepsilon/h}u) +
    e^{\varphi_\varepsilon/h} h^2 A\cdot D(e^{-\varphi_\varepsilon/h}u)\\
    &\quad + e^{\varphi_\varepsilon/h} h^2 D\cdot( A e^{-\varphi_\varepsilon/h}u)
    + e^{\varphi_\varepsilon/h} h^2 \nabla \cdot F
    (e^{-\varphi_\varepsilon/h} u) \big \|_{H_{\textrm{scl}}^{-1}(\R^n)} \\
    & \geq \frac{h}{C}\|u\|_{H^1_{\textrm{scl}}(\R^n)}.
\end{align*}
Moreover we have  that 
\begin{align*}
 \|h^2(A^2+p) u \|_{H_{\textrm{scl}}^{-1}(\R^n)} 
   \leq \mathcal{O}(h^2)\|u\|_{H^1_{\textrm{scl}}(\R^n)}.
\end{align*}
Combining the two previous estimates gives then  that 
\begin{align*}
C\|e^{\varphi_\varepsilon/h}h^2L_{A,q}(e^{-\varphi_\varepsilon/h}u) \|_{H^{-1}_{\textrm{scl}}(\R^n)}
\geq \frac{h}{C} \|u\|_{H^{1}_{\textrm{scl}}(\R^n)},
\end{align*}
where $C>0$. By using $e^{-\varphi_\varepsilon/h}u = e^{-\varphi_/h} e^{-\varphi^2/(2\varepsilon)}u$,
we obtain \eqref{eq:CE_schr}.
\end{proof}

\section*{Acknowledgements}  
The author would like to thank Katya Krupchyk and Pedro Caro for some very useful
discussions.


\end{document}